\documentclass[english]{elsarticle}
\usepackage[T1]{fontenc}
\usepackage[latin9]{inputenc}
\usepackage{prettyref}
\usepackage{amsmath}
\usepackage{amsthm}
\usepackage{amssymb}
\usepackage{nameref}

\makeatletter
\numberwithin{equation}{section}
\numberwithin{figure}{section}
\theoremstyle{plain}
\newtheorem{thm}{\protect\theoremname}[section]
  \theoremstyle{definition}
  \newtheorem{defn}[thm]{\protect\definitionname}
  \theoremstyle{plain}
  \newtheorem{lem}[thm]{\protect\lemmaname}
  \theoremstyle{plain}
  \newtheorem{prop}[thm]{\protect\propositionname}
  \theoremstyle{plain}
  \newtheorem{cor}[thm]{\protect\corollaryname}
  \theoremstyle{definition}
  \newtheorem{example}[thm]{\protect\examplename}
  \theoremstyle{remark}
  \newtheorem{rem}[thm]{\protect\remarkname}

\usepackage{mathtools}
\usepackage{braket}
\usepackage{prettyref}

\newrefformat{prop}{Proposition \ref{#1}}
\newrefformat{cor}{Corollary \ref{#1}}
\newrefformat{exa}{Example \ref{#1}}
\newrefformat{rem}{Remark \ref{#1}}
\newrefformat{def}{Definition \ref{#1}}
\newrefformat{subsec}{Subsection \ref{#1}}

\makeatother

\usepackage{babel}
  \providecommand{\corollaryname}{Corollary}
  \providecommand{\definitionname}{Definition}
  \providecommand{\examplename}{Example}
  \providecommand{\lemmaname}{Lemma}
  \providecommand{\propositionname}{Proposition}
  \providecommand{\remarkname}{Remark}
\providecommand{\theoremname}{Theorem}

\usepackage[
    type={CC},
    modifier={by-nc-nd},
    version={4.0},
]{doclicense}

\begin{document}

\begin{frontmatter}{}

\title{Nonstandard methods in large-scale topology\tnoteref{lic}}
\tnotetext[lic]{\copyright~2019. This manuscript version is made available under the \doclicenseLongNameRef.}
\author{Takuma Imamura}

\address{Research Institute for Mathematical Sciences\\
Kyoto University\\
Kitashirakawa Oiwake-cho, Sakyo-ku, Kyoto 606-8502, Japan}

\ead{timamura@kurims.kyoto-u.ac.jp}
\begin{abstract}
We develop some nonstandard techniques for bornological and coarse
spaces. We first generalise the notion of bornology to prebornology,
which better fits to coarse spaces. We then give nonstandard characterisations
of some basic large-scale notions in terms of galaxies and finite
closeness relations, concepts that have been developed for metric
spaces. Some hybrid notions that involve both small-scale and large-scale
are also discussed. Finally we illustrate an application of our nonstandard
characterisations to prove some elementary facts in large-scale topology
and functional analysis, e.g., the fact that the class of Higson functions
forms a $C^{\ast}$-algebra.
\end{abstract}
\begin{keyword}
bornological space\sep prebornological space\sep coarse space\sep
nonstandard analysis\MSC[2010] 54J05 \sep 46A08 \sep 53C23
\end{keyword}

\end{frontmatter}{}

\section{Introduction}

Small-scale topology is the study of small-scale (fine) structures
of various spaces such as topological spaces and uniform spaces. In
contrast, large-scale topology is the study of the large-scale (coarse)
structures of various spaces such as bornological spaces and coarse
spaces (see Bourbaki \cite{Bou07} and Hogbe-Nlend \cite{HN77} for
bornologies and Roe \cite{Roe93,Roe03} for coarse structures). Nonstandard
analysis was developed by Robinson \cite{Rob66} in the early 1960s.
Since then, it has been successfully applied to various areas of mathematics.
While most research focuses on small-scale concepts (e.g. \cite{KK94,Rob66,SL76}),
little effort has been devoted to a systematic study of large-scale
concepts. An exception is Khalfallah and Kosarew \cite{KK16} which
includes an abstract study of some large-scale concepts via nonstandard
analysis. Other applications can be found in van den Dries and Wilkie
\cite{DW84}, where they uses nonstandard analysis to construct a
special metric space nowadays called the asymptotic cone.

The aim of this paper is to further develop nonstandard treatments
of bornological and coarse spaces. In \prettyref{sec:Bornology} we
discuss the nonstandard treatment of bornological spaces in more detail.
In order to clarify the connection between bornology and coarse structure,
we introduce a (standard) notion of prebornology, a generalisation
of bornology, and then deal with it nonstandardly. In \prettyref{sec:Coarse-structure},
we extend nonstandard methods to the study of coarse spaces. For these
purposes, we start with generalising two nonstandard notions in metric
spaces, galaxy and finite closeness, to prebornological and coarse
spaces. We provide nonstandard characterisations for some large-scale
notions such as bornological maps, proper maps and bornologous maps.
We also discuss some hybrid notions that involve both small-scale
and large-scale, such as local compactness and the slow oscillation
property. We illustrate an application of our nonstandard characterisations
to prove some elementary facts in large-scale topology and functional
analysis, e.g., the fact that the class of Higson functions forms
a $C^{\ast}$-algebra.

\subsection{Basic setting and notation}

We refer to \cite{HN77} for bornology, \cite{Roe03} for coarse topology
and \cite{KK94,Rob66,SL76} for nonstandard analysis and topology.

We work within the Robinson-style framework of nonstandard analysis,
although our methods can be transferred to any other frameworks of
nonstandard analysis. We fix a transitive universe $\mathbb{U}$,
called the standard universe, that satisfies sufficiently many (though
finitely many) axioms of ZFC and contains all standard objects we
consider. We also fix a $\left|\mathbb{U}\right|^{+}$-saturated enlargement
$\ast\colon\mathbb{U}\hookrightarrow\prescript{\ast}{}{\mathbb{U}}$.
The term ``transfer'' refers to the elementary extension property,
``weak saturation'' to the enlargement property, and ``saturation''
to the saturation property.

Let $\left(X,\mathcal{T}_{X}\right)$ be a topological space. The
\emph{monad} of $x\in X$ is the set $\mu_{X}\left(x\right)=\bigcap\set{\prescript{\ast}{}{U}|x\in U\in\mathcal{T}_{X}}$.
The elements of $\mathrm{NS}\left(X\right)=\bigcup_{x\in X}\mu_{X}\left(x\right)$
are called \emph{nearstandard points}. The nonstandard points that
are not nearstandard are called \emph{remote points}. Next, let $\left(Y,\mathcal{U}_{Y}\right)$
be a uniform space. We say that two points $x,y\in\prescript{\ast}{}{Y}$
are \emph{infinitely close} (and write $x\approx_{Y}y$) if for each
$U\in\mathcal{U}_{Y}$ we have that $\left(x,y\right)\in\prescript{\ast}{}{U}$.
Equivalently, the \emph{infinite closeness relation} $\approx_{Y}$
is a binary relation on $\prescript{\ast}{}{Y}$ defined as the intersection
$\bigcap_{U\in\mathcal{U}_{Y}}\prescript{\ast}{}{U}$.

\section{\label{sec:Bornology}Bornology}

In this section, we discuss the nonstandard treatment of (pre)bornological
spaces. We provide nonstandard characterisations for some large-scale
notions concerning bornology. Using these characterisations, we prove
some well-known facts in large-scale topology and functional analysis.
The interaction between topology and bornology is also discussed.

\subsection{Bornologies and galaxy maps}
\begin{defn}[Standard]
A \emph{prebornology} on a set $X$ is a nonempty subset $\mathcal{B}_{X}$
of $\mathcal{P}\left(X\right)$ such that
\begin{enumerate}
\item $\mathcal{B}_{X}$ is a cover of $X$: $\bigcup\mathcal{B}_{X}=X$;
\item $\mathcal{B}_{X}$ is downward closed: $B\in\mathcal{B}_{X}$ and
$C\subseteq B$ implies $C\in\mathcal{B}_{X}$;
\item $\mathcal{B}_{X}$ is closed under finite \emph{non-disjoint} union:
$B,C\in\mathcal{B}_{X}$ and $B\cap C\neq\varnothing$ implies $B\cup C\in\mathcal{B}_{X}$.
\end{enumerate}
The pair $\left(X,\mathcal{B}_{X}\right)$ is called a \emph{prebornological
space}, and the elements of $\mathcal{B}_{X}$ are called \emph{bounded
sets} of $X$. For $x\in X$ we denote $\mathcal{BN}_{X}\left(x\right)=\set{B\in\mathcal{B}_{X}|x\in B}$.
\end{defn}
For instance, the collection of all bounded subsets of a metric space
in the usual sense is a prebornology. More examples are found in \prettyref{subsec:Examples-of-bornology}.

The notion of prebornology is a slight generalisation of bornology.
Recall the definition of the latter.
\begin{defn}[Standard]
A \emph{bornology} on a set $X$ is a cover of $X$ which is downward
closed and is closed under finite \emph{possibly disjoint} union.
A set with a bornology is called a \emph{bornological space}.
\end{defn}
A problem of this definition is that the boundedness induced by a
coarse structure may not be a bornology, while it is a prebornology.
Because of this, prebornology is more suitable than bornology when
considering coarse structures (see also \prettyref{rem:Adjoint}).

It is well-known that topology has a pointwise definition involving
neighbourhood systems. Similarly, prebornology also has a pointwise
definition via bornological neighbourhood systems.
\begin{lem}[Standard]
\label{lem:local-def-of-bornology}Let $X$ be a set. For each $x\in X$,
let $\mathcal{BN}\left(x\right)$ be a nonempty subset of $\mathcal{P}\left(X\right)$
that satisfies the following axioms:
\begin{enumerate}
\item[(BN1)]  if $B\in\mathcal{BN}\left(x\right)$, then $x\in B$;
\item[(BN2)] if $x\in A\subseteq B\in\mathcal{BN}\left(x\right)$, then $A\in\mathcal{BN}\left(x\right)$;
\item[(BN3)] if $A,B\in\mathcal{BN}\left(x\right)$, then $A\cup B\in\mathcal{BN}\left(x\right)$;
\item[(BN4)] if $B\in\mathcal{BN}\left(x\right)$, then $B\in\mathcal{BN}\left(y\right)$
for each $y\in B$.
\end{enumerate}
Then, there is a unique prebornology on $X$ such that $\mathcal{BN}_{X}$
coincides with the given $\mathcal{BN}$. Conversely, given a prebornology
on $X$, its $\mathcal{BN}_{X}$ satisfies (BN1) to (BN4).

\end{lem}
\begin{proof}
We here only prove the former part. Let $\mathcal{B}_{X}=\bigcup_{x\in X}\mathcal{BN}\left(x\right)\cup\set{\varnothing}$.
First, $\mathcal{B}_{X}$ covers $X$ by non-emptiness of $\mathcal{BN}\left(x\right)$
and (BN1). The downward closedness of $\mathcal{B}_{X}$ follows from
(BN2). Now let $B,C\in\mathcal{B}_{X}$ and suppose that $x_{0}\in B\cap C\neq\varnothing$.
There are $y,z\in X$ such that $B\in\mathcal{BN}\left(y\right)$
and $C\in\mathcal{BN}\left(z\right)$. By (BN4), $B,C\in\mathcal{BN}\left(x_{0}\right)$.
By (BN3), $B\cup C\in\mathcal{BN}\left(x_{0}\right)$, so $B\cup C\in\mathcal{B}_{X}$.
Hence $\mathcal{B}_{X}$ is a prebornology on $X$. It is easy to
see that $\mathcal{BN}\left(x\right)=\set{B\in\mathcal{B}_{X}|x\in B}$.
\end{proof}
\begin{defn}
Let $X$ be a prebornological space. The \emph{galaxy} of $x\in X$
is defined as follows (Robinson \cite[p.101]{Rob66} for the metric
case):
\[
G_{X}\left(x\right)=\bigcup_{B\in\mathcal{BN}_{X}\left(x\right)}\prescript{\ast}{}{B}.
\]
\end{defn}
The prebornology can be recovered from the galaxy map $G_{X}\colon X\to\mathcal{P}\left(\prescript{\ast}{}{X}\right)$.
To see this, we will use the following lemma, a slight generalisation
of \cite[Lemma 4.4]{KK16} to prebornology.
\begin{lem}[Bornological Approximation Lemma]
\label{lem:bornological-approximation-lemma}Let $X$ be a prebornological
space. For each $x\in X$, there exists a $B\in\prescript{\ast}{}{\mathcal{B}_{X}}$
such that $G_{X}\left(x\right)\subseteq B$.
\end{lem}
\begin{proof}
Since $\mathcal{BN}_{X}\left(x\right)$ is closed under finite union,
for every finite subset $\mathcal{A}$ of $\mathcal{BN}_{X}\left(x\right)$
there exists an $B'\in\mathcal{BN}_{X}\left(x\right)$ such that $C\subseteq B'$
for all $C\in\mathcal{A}$. By weak saturation, there exists an $B\in\prescript{\ast}{}{\mathcal{BN}_{X}\left(x\right)}$
such that $\prescript{\ast}{}{C}\subseteq B$ holds for all $B\in\mathcal{BN}_{X}\left(x\right)$.
Hence $G_{X}\left(x\right)\subseteq B$.
\end{proof}
\begin{prop}
\label{prop:nonst-characterisation-of-boundedness}Let $X$ be a prebornological
space and let $B$ be a subset of $X$. The following are equivalent:
\begin{enumerate}
\item $B$ is bounded;
\item $\prescript{\ast}{}{B}\subseteq G_{X}\left(x\right)$ for all $x\in B$;
\item $B=\varnothing$ or $\prescript{\ast}{}{B}\subseteq G_{X}\left(x\right)$
for some $x\in B$;
\item $X=\varnothing$ or $\prescript{\ast}{}{B}\subseteq G_{X}\left(x\right)$
for some $x\in X$.
\end{enumerate}
\end{prop}
\begin{proof}
The only nontrivial part is $\left(4\right)\Rightarrow\left(1\right)$.
Suppose that $\prescript{\ast}{}{B}\subseteq G\left(x\right)$ for
some $x\in X$. By \nameref{lem:bornological-approximation-lemma},
there exists a $C\in\prescript{\ast}{}{\mathcal{B}_{X}}$ such that
$G_{X}\left(x\right)\subseteq C$. Since $\prescript{\ast}{}{B}\subseteq G\left(x\right)\subseteq C$,
$B$ is bounded by transfer.
\end{proof}
\begin{defn}
A set $A$ is said to be \emph{galactic} if $A=\bigcup\set{\prescript{\ast}{}{B}|\prescript{\ast}{}{B}\subseteq A}$.
\end{defn}

\begin{prop}
\label{prop:properties-of-galaxy-map}Let $X$ be a prebornological
space. The map $G_{X}\colon X\to\mathcal{P}\left(\prescript{\ast}{}{X}\right)$
satisfies the following:
\begin{enumerate}
\item $G_{X}$ is pointwise galactic, that is, $G_{X}\left(x\right)$ is
galactic for any $x\in X$;
\item $x\in G_{X}\left(x\right)$ for any $x\in X$;
\item $G_{X}\left(x\right)\cap G_{X}\left(y\right)\neq\varnothing\iff G_{X}\left(x\right)=G_{X}\left(y\right)$
for any $x,y\in X$.
\end{enumerate}
\end{prop}
\begin{proof}
(1) is clear by definition. (2) is immediate from the fact that $\mathcal{B}_{X}$
is a cover of $X$. To see (3), suppose that $G_{X}\left(x\right)\cap G_{X}\left(y\right)\neq\varnothing$.
Fix a point $z\in G_{X}\left(x\right)\cap G_{X}\left(y\right)$. There
are bounded sets $B_{x},B_{y}$ such that $z,x\in\prescript{\ast}{}{B_{x}}$
and $z,y\in\prescript{\ast}{}{B_{y}}$. Since $z\in\prescript{\ast}{}{B_{x}}\cap\prescript{\ast}{}{B_{y}}\neq\varnothing$,
we have $B_{x}\cap B_{y}\neq\varnothing$ by transfer. Hence $B_{x}\cup B_{y}$
is bounded. Let $t\in G_{X}\left(x\right)$. There exists a bounded
set $B_{t}$ such that $t,x\in\prescript{\ast}{}{B_{t}}$. Since $x\in B_{t}\cap\left(B_{x}\cup B_{y}\right)\neq\varnothing$,
$B_{t}\cup B_{x}\cup B_{y}$ is bounded. Hence $t\in\prescript{\ast}{}{\left(B_{t}\cup B_{x}\cup B_{y}\right)}\subseteq G_{X}\left(y\right)$.
Similarly, if $t\in G_{X}\left(y\right)$, then $t\in G_{X}\left(x\right)$.
The reverse direction is immediate from (2).
\end{proof}
\begin{thm}
\label{thm:external-def-of-prebornology}Let $X$ be any set. If a
map $G\colon X\to\mathcal{P}\left(\prescript{\ast}{}{X}\right)$ satisfies
(1) to (3) in \prettyref{prop:properties-of-galaxy-map}, then $X$
admits a unique prebornology whose galaxy map coincides with $G$.
\end{thm}
\begin{proof}
For $x\in X$, let $\mathcal{BN}\left(x\right)=\set{B\subseteq X|x\in B\text{ and }\prescript{\ast}{}{B}\subseteq G\left(x\right)}$.
To apply \prettyref{lem:local-def-of-bornology}, we show that $\mathcal{BN}$
satisfies the axioms (BN1) to (BN4). (BN1) follows from (2). (BN2)
and (BN3) are immediate from the definition of $\mathcal{BN}$ and
the transfer principle. To prove (BN4), let $B\in\mathcal{BN}\left(x\right)$
and $y\in B$. Then, $y\in\prescript{\ast}{}{B}\subseteq G\left(x\right)$.
On the other hand, $y\in G\left(y\right)$ by (2). Hence $G\left(x\right)\cap G\left(y\right)\neq\varnothing$.
By (3), $G\left(x\right)=G\left(y\right)$, so $B\in\mathcal{BN}\left(y\right)$.

Applying \prettyref{lem:local-def-of-bornology} we obtain a unique
prebornology on $X$ whose $\mathcal{BN}_{X}$ coincides with $\mathcal{BN}$.
Now, let $x\in X$. Since $G\left(x\right)$ is galactic, it follows
that
\begin{align*}
G\left(x\right) & =\bigcup\set{\prescript{\ast}{}{B}|\prescript{\ast}{}{B}\subseteq G\left(x\right)}\\
 & =\bigcup_{B\in\mathcal{BN}\left(x\right)}\prescript{\ast}{}{B}\\
 & =\bigcup_{B\in\mathcal{BN}_{X}\left(x\right)}\prescript{\ast}{}{B}\\
 & =G_{X}\left(x\right).
\end{align*}
\end{proof}
Hence the correspondence $\mathcal{B}_{X}\leftrightarrow G_{X}$ gives
a bijection between the prebornologies on $X$ and the maps $X\to\mathcal{P}\left(\prescript{\ast}{}{X}\right)$
satisfying (1) to (3). 
\begin{defn}[Standard]
A prebornological space $X$ is \emph{(bornologically) connected}
if every finite subset of $X$ is bounded.
\end{defn}
A bornological space is precisely a connected prebornological space
in our terminology.
\begin{prop}
\label{prop:characterisation-of-bornological-connectedness}Let $X$
be a prebornological space. The following are equivalent:
\begin{enumerate}
\item $X$ is connected;
\item $X$ is a bornological space;
\item $G_{X}\left(x\right)\supseteq X$ for all $x\in X$;
\item $G_{X}\left(x\right)=G_{X}\left(y\right)$ for all $x,y\in X$.
\end{enumerate}
\end{prop}
\begin{proof}
Let us only prove the equivalence between the standard definition
$\left(1\right)$ and the nonstandard one $\left(3\right)$.

$\left(1\right)\Rightarrow\left(3\right)$: $G_{X}\left(x\right)\supseteq\bigcup_{y\in X}\prescript{\ast}{}{\set{x,y}}=\bigcup_{y\in X}\set{x,y}=X$.

$\left(3\right)\Rightarrow\left(1\right)$: Let $A$ be a finite subset
of $X$. By assumption, $\prescript{\ast}{}{A}=A\subseteq X\subseteq G_{X}\left(x\right)$
holds for all $x\in X$. By \prettyref{prop:nonst-characterisation-of-boundedness},
$A$ is bounded.
\end{proof}
\begin{cor}
\label{cor:external-def-of-bornology}Let $X$ be any set. For each
galactic subset $G$ of $\prescript{\ast}{}{X}$ that contains $X$,
there is a unique bornology such that $G_{X}\left(x\right)=G$ holds
for all $x\in X$.
\end{cor}

\subsection{\label{subsec:Examples-of-bornology}Examples of (pre)bornological
spaces}
\begin{example}
Let $X$ be a set. The \emph{maximal bornology} on $X$ is the power
set $\mathcal{P}\left(X\right)=\set{A|A\subseteq X}$, i.e., all subsets
are bounded. The galaxy map is identically $G_{X}\left(x\right)=\prescript{\ast}{}{X}$.
\end{example}

\begin{example}
Let $X$ be a set. The \emph{discrete prebornology} on $X$ is $\set{\varnothing}\cup\set{\set{x}|x\in X}$.
The galaxy map is $G_{X}\left(x\right)=\set{x}$. This prebornology
is connected (so a bornology) only when $\left|X\right|\leq1$.
\end{example}

\begin{example}
Let $X$ be a set. The \emph{finite bornology} on $X$ is $\set{A\subseteq X|A\text{ is finite}}$.
In this space, the bounded sets are precisely the finite subsets.
The galaxy map is identically $G_{X}\left(x\right)=X$.
\end{example}

\begin{example}
Let $X$ be a topological space. The \emph{compact bornology} of $X$
is
\[
\mathcal{P}_{c}\left(X\right)=\set{A\subseteq X|A\text{ is contained in some compact set}}.
\]
The galaxy map satisfies that $G_{X}\left(x\right)\subseteq\mathrm{NS}\left(X\right)$
for all $x\in X$. To see this, suppose that $A\subseteq X$ contains
$x$ and is contained in a compact set $K$. By the nonstandard characterisation
of compactness (cf. \cite[Theorem 4.1.13]{Rob66}), $\prescript{\ast}{}{A}\subseteq\prescript{\ast}{}{K}\subseteq\mathrm{NS}\left(K\right)\subseteq\mathrm{NS}\left(X\right)$.
Hence $G_{X}\left(x\right)\subseteq\mathrm{NS}\left(X\right)$.

If $X$ is $T_{1}$, we can consider another bornology on $X$, called
the \emph{relatively compact bornology}, defined as follows:
\[
\mathcal{P}_{rc}\left(X\right)=\set{A\subseteq X|A\text{ is relatively compact}}.
\]
Since $\mathcal{P}_{rc}\left(X\right)\subseteq\mathcal{P}_{c}\left(X\right)$,
its galaxy map also satisfies that $G_{X}\left(x\right)\subseteq\mathrm{NS}\left(X\right)$
for all $x\in X$.
\end{example}

\begin{example}
\label{exa:bounded-bornology} Let $X$ be a pseudometric space. The \emph{bounded bornology} of
$X$ is
\[
\set{B\subseteq X|B\text{ is bounded in the usual sense}}.
\]
The galaxy map is $G_{X}\left(x\right)=\set{y\in\prescript{\ast}{}{X}|\prescript{\ast}{}{d_{X}}\left(x,y\right)\text{ is finite}}$,
which is independent of $x$. This construction works even when the
pseudometric function is allowed to take the value $+\infty$. In
this case, the resulting prebornology may be disconnected. It is connected
when and only when the pseudometric function is finite-valued.
\end{example}

\begin{example}
Let $X$ be a prebornological space and $A$ a subset of $X$. The
\emph{subspace prebornology} of $A$ is $\mathcal{B}_{X}\restriction A=\set{B\cap A|B\in\mathcal{B}_{X}}$.
The galaxy map is $G_{A}\left(x\right)=G_{X}\left(x\right)\cap\prescript{\ast}{}{A}$.
If $X$ is connected, then so is $A$.
\end{example}

\begin{example}
Let $\set{\mathcal{B}_{i}}_{i\in I}$ be a family of prebornologies
on a fixed set $X$. The intersection $\bigcap_{i\in I}\mathcal{B}_{i}$
forms a prebornology on $X$. Its galaxy map is
\[
G_{X}\left(x\right)=\bigcup\Set{\prescript{\ast}{}{B}|x\in B\subseteq X\text{ and }\prescript{\ast}{}{B}\subseteq\bigcap_{i\in I}G_{i}\left(x\right)},
\]
where $G_{i}$ is the galaxy map for $\mathcal{B}_{i}$. The right
hand side may not be equal to $\bigcap_{i\in I}G_{i}\left(x\right)$,
because $\bigcap_{i\in I}G_{i}\left(x\right)$ may not be galactic.
However, if $I$ is finite, then $G_{X}\left(x\right)=\bigcap_{i\in I}G_{i}\left(x\right)$.
If each $\mathcal{B}_{i}$ is connected, then so is $\bigcap\mathcal{B}_{i}$.
\end{example}

\begin{example}
Given a prebornology $\mathcal{B}_{X}$ on a set $X$, the smallest
bornology $\mathcal{B}_{X}'\supseteq\mathcal{B}_{X}$ on $X$ exists.
Its galaxy map is given by $G_{X}'\left(x\right)=\bigcup_{y\in X}G_{X}\left(y\right)$,
where $G_{X}$ and $G_{X}'$ denote the galaxy maps for $\mathcal{B}_{X}$
and $\mathcal{B}_{X}'$, respectively.
\end{example}

\begin{example}
Let $\set{X_{i}}_{i\in I}$ be a family of prebornological spaces.
The \emph{product prebornology} on $P=\prod_{i\in I}X_{i}$ is 
\[
\Set{B\subseteq P|\pi_{i}\left(B\right)\in\mathcal{B}_{X_{i}}\text{ for all }i\in I},
\]
where $\pi_{i}\colon P\to X_{i}$ is the canonical projection. The
galaxy map is 
\[
G_{P}\left(x\right)=\bigcup\Set{\prescript{\ast}{}{B}|x\in B\subseteq P\text{ and }\prescript{\ast}{}{B}\subseteq\prod_{i\in I}G_{X_{i}}\left(x_{i}\right)\times\prod_{j\in\prescript{\ast}{}{I}\setminus I}\prescript{\ast}{}{X_{j}}}.
\]
If $I$ is finite, the remainder $\prescript{\ast}{}{I}\setminus I$
vanishes and $G_{P}\left(x\right)=\prod_{i\in I}G_{X_{i}}\left(x_{i}\right)$
holds. If each $X_{i}$ is connected, then so is $P$.
\end{example}

\begin{example}
Let $\set{X_{i}}_{i\in I}$ be a family of prebornological spaces.
The \emph{sum prebornology} on $S=\coprod_{i\in I}X_{i}$ is $\bigcup_{i\in I}\mathcal{B}_{X_{i}}$.
The galaxy map is $G_{S}\left(x\right)=G_{X_{i}}\left(x\right)$,
$x\in X_{i}$. If $\set{X_{i}}_{i\in I}$ contains at least two nonempty
spaces, the sum $S$ is disconnected.
\end{example}

\subsection{Bornological maps and proper maps}
\begin{defn}[Standard]
Let $X$ and $Y$ be prebornological spaces. A map $f\colon X\to Y$
is \emph{bornological} at $x\in X$ if the direct image $f$ maps
$\mathcal{BN}_{X}\left(x\right)$ to $\mathcal{BN}_{Y}\left(f\left(x\right)\right)$.
\end{defn}
This notion admits the following nonstandard characterisation which
is an obvious generalisation of \cite[Theorem 4.5]{KK16} to prebornology.
\begin{thm}
\label{thm:nonst-characterisation-of-bornological-map}Let $X$ and
$Y$ be prebornological spaces and let $f\colon X\to Y$ be a map.
Consider a point $x\in X$. The following are equivalent:
\begin{enumerate}
\item $f$ is bornological at $x$;
\item $\prescript{\ast}{}{f}\left(G_{X}\left(x\right)\right)\subseteq G_{Y}\left(f\left(x\right)\right)$.
\end{enumerate}
\end{thm}
\begin{proof}
$\left(1\right)\Rightarrow\left(2\right)$: The direct image $f$
maps $\mathcal{BN}_{X}\left(x\right)$ to $\mathcal{BN}_{Y}\left(f\left(x\right)\right)$.
Hence
\begin{align*}
\prescript{\ast}{}{f}\left(G_{X}\left(x\right)\right) & =\prescript{\ast}{}{f}\bigcup_{B\in\mathcal{BN}_{X}\left(x\right)}\prescript{\ast}{}{B}\\
 & =\bigcup_{B\in\mathcal{BN}_{X}\left(x\right)}\prescript{\ast}{}{f}\left(\prescript{\ast}{}{B}\right)\\
 & \subseteq\bigcup_{C\in\mathcal{BN}_{Y}\left(f\left(x\right)\right)}\prescript{\ast}{}{C}\\
 & =G_{Y}\left(f\left(x\right)\right).
\end{align*}

$\left(2\right)\Rightarrow\left(1\right)$: Let $B\in\mathcal{BN}_{X}\left(x\right)$.
By \nameref{lem:bornological-approximation-lemma}, choose a $C\in\prescript{\ast}{}{\mathcal{BN}_{Y}\left(f\left(x\right)\right)}$
such that $G_{Y}\left(f\left(x\right)\right)\subseteq C$. Then $\prescript{\ast}{}{f}\left(\prescript{\ast}{}{B}\right)\subseteq\prescript{\ast}{}{f}\left(G_{X}\left(x\right)\right)\subseteq G_{Y}\left(f\left(x\right)\right)\subseteq C\in\prescript{\ast}{}{\mathcal{BN}_{Y}\left(f\left(x\right)\right)}$.
By transfer, we have that $f\left(B\right)\in\mathcal{BN}_{Y}\left(f\left(x\right)\right)$.
\end{proof}
\begin{defn}[Standard]
Let $X$ and $Y$ be prebornological spaces. A map $f\colon X\to Y$
is said to be \emph{proper} if the inverse image $f^{-1}$ maps $\mathcal{B}_{Y}$
to $\mathcal{B}_{X}$.
\end{defn}
\begin{thm}
\label{thm:nonst-characterisation-of-proper-map}Let $X$ and $Y$
be prebornological spaces and let $f\colon X\to Y$ be a map. The
following are equivalent:

\begin{enumerate}
\item $f$ is proper;
\item for every $x\in X$, $\prescript{\ast}{}{f}^{-1}\left(G_{Y}\left(f\left(x\right)\right)\right)\subseteq G_{X}\left(x\right)$.
\end{enumerate}
\end{thm}
\begin{proof}
$\left(1\right)\Rightarrow\left(2\right)$: For each $B\in\mathcal{BN}_{Y}\left(f\left(x\right)\right)$,
by \prettyref{prop:nonst-characterisation-of-boundedness}, $\prescript{\ast}{}{f}^{-1}\left(\prescript{\ast}{}{B}\right)\subseteq G_{X}\left(x\right)$
holds. We have that
\[
\prescript{\ast}{}{f}^{-1}\left(G_{Y}\left(f\left(x\right)\right)\right)=\bigcup_{B\in\mathcal{BN}_{Y}\left(f\left(x\right)\right)}\prescript{\ast}{}{f}^{-1}\left(\prescript{\ast}{}{B}\right)\subseteq G_{X}\left(x\right).
\]

$\left(2\right)\Rightarrow\left(1\right)$: Let $B\in\mathcal{B}_{Y}$.
If $f^{-1}\left(B\right)$ is empty, it is bounded. If not, let $x\in f^{-1}\left(B\right)$.
Since $f\left(x\right)\in B$, we have that $\prescript{\ast}{}{B}\subseteq G_{Y}\left(f\left(x\right)\right)$
by \prettyref{prop:nonst-characterisation-of-boundedness}. By (2),
$\prescript{\ast}{}{f}^{-1}\left(\prescript{\ast}{}{B}\right)\subseteq G_{X}\left(x\right)$
holds. By \prettyref{prop:nonst-characterisation-of-boundedness},
$f^{-1}\left(B\right)$ is bounded.
\end{proof}
\begin{defn}
Let $X$ be a prebornological space. A point $x\in\prescript{\ast}{}{X}$
is said to be \emph{finite} if there exists an $x'\in X$ such that
$x\in G_{X}\left(x'\right)$. A point of $\prescript{\ast}{}{X}$
is said to be \emph{infinite} if it is not finite. We denote the set
of all finite points of $\prescript{\ast}{}{X}$ by $\mathrm{FIN}\left(X\right)$
and the set of all infinite points of $\prescript{\ast}{}{X}$ by
$\mathrm{INF}\left(X\right)$.
\end{defn}
Using this terminology, we can rephrase the characterisations of bornologicity
and properness more simply.
\begin{thm}
\label{thm:nonst-characterisations-of-bornological-and-proper}Let
$X$ and $Y$ be prebornological spaces and let $f\colon X\to Y$
be a map.
\begin{enumerate}
\item If $f$ is bornological, then $\prescript{\ast}{}{f}$ maps $\mathrm{FIN}\left(X\right)$
to $\mathrm{FIN}\left(Y\right)$. If $Y$ is connected, the converse
is also true.
\item If $f$ is bornological, then $\prescript{\ast}{}{f}^{-1}$ maps $\mathrm{INF}\left(Y\right)$
to $\mathrm{INF}\left(X\right)$ as a multi-valued map. If $Y$ is
connected, the converse is also true.
\item If $f$ is proper, then $\prescript{\ast}{}{f}^{-1}$ maps $\mathrm{FIN}\left(Y\right)$
to $\mathrm{FIN}\left(X\right)$ as a multi-valued map. If $X$ is
connected, the converse is also true.
\item If $f$ is proper, then $\prescript{\ast}{}{f}$ maps $\mathrm{INF}\left(X\right)$
to $\mathrm{INF}\left(Y\right)$. If $X$ is connected, the converse
is also true.
\end{enumerate}
\end{thm}
\begin{proof}
(2) and (4) immediately follow from (1) and (3), respectively. We
only prove (1) and (3).
\begin{enumerate}
\item The forward direction is immediate from \prettyref{thm:nonst-characterisation-of-bornological-map}.
Suppose that $Y$ is connected and $\prescript{\ast}{}{f}\left(\mathrm{FIN}\left(X\right)\right)\subseteq\mathrm{FIN}\left(Y\right)$.
Let $x\in X$. Then, $\prescript{\ast}{}{f}\left(G_{X}\left(x\right)\right)\subseteq\prescript{\ast}{}{f}\left(\mathrm{FIN}\left(X\right)\right)$.
By \prettyref{prop:characterisation-of-bornological-connectedness},
$\mathrm{FIN}\left(Y\right)=G_{Y}\left(f\left(x\right)\right)$. Combining
them, we have that $\prescript{\ast}{}{f}\left(G_{X}\left(x\right)\right)\subseteq G_{Y}\left(f\left(x\right)\right)$.
By \prettyref{thm:nonst-characterisation-of-bornological-map}, $f$
is bornological.
\item[3.] The forward direction is immediate from \prettyref{thm:nonst-characterisation-of-proper-map}.
Suppose that $X$ is connected and $\prescript{\ast}{}{f}^{-1}\left(\mathrm{FIN}\left(Y\right)\right)\subseteq\mathrm{FIN}\left(X\right)$.
Let $x\in X$. Then, $\prescript{\ast}{}{f}^{-1}\left(G_{Y}\left(f\left(x\right)\right)\right)\subseteq\prescript{\ast}{}{f}^{-1}\left(\mathrm{FIN}\left(Y\right)\right)\subseteq\mathrm{FIN}\left(X\right)$.
By \prettyref{prop:characterisation-of-bornological-connectedness},
$\mathrm{FIN}\left(X\right)=G_{X}\left(x\right)$ holds. Hence $\prescript{\ast}{}{f}^{-1}\left(G_{Y}\left(f\left(x\right)\right)\right)\subseteq G_{X}\left(x\right)$.
By \prettyref{thm:nonst-characterisation-of-proper-map}, $f$ is
proper.
\end{enumerate}
\end{proof}
Let us demonstrate the use of such characterisations by proving the
following well-known result in elementary functional analysis.
\begin{prop}[Standard]
Let $X$ and $Y$ be normed vector spaces with the bounded bornologies.
For each linear map $f\colon X\to Y$, $f$ is continuous if and only
if $f$ is bornological.
\end{prop}
\begin{proof}
We only need to consider the behaviour around $0$. Firstly, suppose
that $f$ is not bornological at $0$. By \prettyref{thm:nonst-characterisations-of-bornological-and-proper},
we can choose an $x\in\mathrm{FIN}\left(X\right)$ such that $\prescript{\ast}{}{f}\left(x\right)\notin\mathrm{FIN}\left(Y\right)$.
Let $\lambda=\left\Vert \prescript{\ast}{}{f}\left(x\right)\right\Vert ^{-1}$.
Then, $\left\Vert \lambda x\right\Vert =\lambda\left\Vert x\right\Vert =\text{infinitesimal}\times\text{finite}=\text{infinitesimal}$,
so $\lambda x\in\mu_{X}\left(0\right)$. However, $\left\Vert \prescript{\ast}{}{f}\left(\lambda x\right)\right\Vert =\lambda\left\Vert \prescript{\ast}{}{f}\left(x\right)\right\Vert =1$,
so $\prescript{\ast}{}{f}\left(\lambda x\right)\notin\mu_{Y}\left(0\right)$.
By the nonstandard characterisation of continuity (cf. \cite[Theorem 4.2.7]{Rob66}),
$f$ is discontinuous at $0$.

Secondly, suppose that $f$ is discontinuous at $0$. By the nonstandard
characterisation of continuity, we can choose an $x\in\mu_{X}\left(0\right)$
such that $\prescript{\ast}{}{f}\left(x\right)\notin\mu_{Y}\left(0\right)$.
Now, let $\lambda=\left\Vert x\right\Vert ^{-1}$. Since $\left\Vert \lambda x\right\Vert =1$,
we have $\lambda x\in\mathrm{FIN}\left(X\right)$. On the other hand,
$\left\Vert \prescript{\ast}{}{f}\left(\lambda x\right)\right\Vert =\lambda\left\Vert \prescript{\ast}{}{f}\left(x\right)\right\Vert =\text{infinite}\times\text{non-infinitesimal}=\text{infinite}$,
so $\prescript{\ast}{}{f}\left(\lambda x\right)\notin\mathrm{FIN}\left(Y\right)$.
By \prettyref{thm:nonst-characterisations-of-bornological-and-proper},
$f$ is not bornological.
\end{proof}

\subsection{Simply bounded and equibounded families}
\begin{defn}[Standard]
Let $X$ and $Y$ be prebornological spaces. We denote by $Y^{X}$
the set of all maps from $X$ to $Y$. Let $\mathcal{F}$ be a subset
of $Y^{X}$. For $B\subseteq X$, we denote $\mathcal{F}\left(B\right)=\set{f\left(x\right)|f\in\mathcal{F}\text{ and }x\in B}$.
$\mathcal{F}$ is said to be \emph{simply bounded} if $\mathcal{F}\left(x\right)\in\mathcal{B}_{Y}$
for all $x\in X$. $\mathcal{F}$ is said to be \emph{equibounded}
if $\mathcal{F}\left(B\right)\in\mathcal{B}_{Y}$ for all $B\in\mathcal{B}_{X}$.
\end{defn}
\begin{thm}
Let $X$ and $Y$ be prebornological space and let $\mathcal{F}\subseteq Y^{X}$
be nonempty. The following are equivalent:
\begin{enumerate}
\item $\mathcal{F}$ is simply bounded;
\item for every $x\in X$, there exists a $y\in Y$ such that $\prescript{\ast}{}{\mathcal{F}}\left(x\right)\subseteq G_{Y}\left(y\right)$.
\end{enumerate}
\end{thm}
\begin{proof}
Immediate from \prettyref{prop:nonst-characterisation-of-boundedness}.
\end{proof}
\begin{cor}
\label{cor:nonst-char-simply-bounded}Let $X$ and $Y$ be prebornological
spaces and let $\mathcal{F}\subseteq Y^{X}$. If $\mathcal{F}$ is
simply bounded, then $\prescript{\ast}{}{\mathcal{F}}$ maps $X$
to $\mathrm{FIN}\left(Y\right)$ as a multi-valued map. If $Y$ is
connected, the converse is also true.
\end{cor}
\begin{thm}
Let $X$ and $Y$ be prebornological spaces and let $\mathcal{F}\subseteq Y^{X}$
be nonempty. The following are equivalent:
\begin{enumerate}
\item $\mathcal{F}$ is equibounded;
\item for every $x\in X$, there exists a $y\in Y$ such that $\prescript{\ast}{}{\mathcal{F}}\left(G_{X}\left(x\right)\right)\subseteq G_{Y}\left(y\right)$.
\end{enumerate}
\end{thm}
\begin{proof}
$\left(1\right)\Rightarrow\left(2\right)$: Let $x\in X$. We fix
an $f\in\mathcal{F}$. Let $B\in\mathcal{BN}_{X}\left(x\right)$.
Since $\mathcal{F}\left(B\right)$ is bounded and $f\left(x\right)\in\mathcal{F}\left(B\right)$,
$\prescript{\ast}{}{\mathcal{F}}\left(\prescript{\ast}{}{B}\right)\subseteq G_{Y}\left(f\left(x\right)\right)$
holds by \prettyref{prop:nonst-characterisation-of-boundedness}.
Hence $\prescript{\ast}{}{\mathcal{F}}\left(G_{X}\left(x\right)\right)\subseteq G_{Y}\left(f\left(x\right)\right)$,
because $B$ was arbitrary.

$\left(2\right)\Rightarrow\left(1\right)$: Let $B$ be a bounded
set of $X$. We may assume that $B\neq\varnothing$. We fix an $x_{0}\in B$.
We can find a $y\in Y$ such that $\prescript{\ast}{}{\mathcal{F}}\left(G_{X}\left(x_{0}\right)\right)\subseteq G_{Y}\left(y\right)$.
Then, $\prescript{\ast}{}{\mathcal{F}}\left(\prescript{\ast}{}{B}\right)\subseteq G_{Y}\left(y\right)$.
By \prettyref{prop:nonst-characterisation-of-boundedness}, $\mathcal{F}\left(B\right)$
is bounded in $Y$. Therefore $\mathcal{F}$ is equibounded.
\end{proof}
\begin{cor}
\label{cor:nonst-char-equibounded}Let $X$ and $Y$ be prebornological
spaces and let $\mathcal{F}\subseteq Y^{X}$. If $\mathcal{F}$ is
equibounded, then $\prescript{\ast}{}{\mathcal{F}}$ maps $\mathrm{FIN}\left(X\right)$
to $\mathrm{FIN}\left(Y\right)$ as a multi-valued map. If $Y$ is
connected, the converse is also true.
\end{cor}
\begin{prop}[Standard]\label{prop:equibounded-family-of-linear-maps}
Let $X$ and $Y$ be normed vector spaces with the bounded bornologies.
For each $\mathcal{F}\subseteq Y^{X}$, $\mathcal{F}$ is equicontinuous
if and only if $\mathcal{F}$ is equibounded.
\end{prop}
\begin{proof}
We only need to consider the behaviour around $0$. First, suppose
that $\mathcal{F}$ is not equibounded. By \prettyref{cor:nonst-char-equibounded},
we can choose an $x\in\mathrm{FIN}\left(X\right)$ and an $f\in\prescript{\ast}{}{\mathcal{F}}$
such that $f\left(x\right)\in\mathrm{INF}\left(Y\right)$. Let $\lambda=\left\Vert f\left(x\right)\right\Vert ^{-1}$.
Then, $\lambda x\in\mu_{X}\left(0\right)$ but $f\left(\lambda x\right)\notin\mu_{Y}\left(0\right)$.
By the nonstandard characterisation of equicontinuity (cf. \cite[4.4.6]{KK94}),
$\mathcal{F}$ is not equicontinuous at $0$.

Next, suppose that $\mathcal{F}$ is not equicontinuous at $0$. By
the nonstandard characterisation of equicontinuity, we can choose
an $x\in\mu_{X}\left(0\right)$ and an $f\in\prescript{\ast}{}{\mathcal{F}}$
such that $f\left(x\right)\notin\mu_{Y}\left(0\right)$. Now, let
$\lambda=\left\Vert x\right\Vert ^{-1}$. Then, $\lambda x\in\mathrm{FIN}\left(X\right)$
but $f\left(\lambda x\right)\notin\mathrm{FIN}\left(Y\right)$. By
\prettyref{cor:nonst-char-equibounded}, $\mathcal{F}$ is not equibounded.
\end{proof}

\subsection{Compatibility of topology and bornology}
\begin{defn}[Standard]
Let $X$ be a set. A topology and a prebornology on $X$ are said
to be \emph{compatible} if $X$ is locally bounded, i.e., each point
of $X$ has a bounded neighbourhood.
\end{defn}
\begin{thm}
\label{thm:compatibility-of-top-and-bor}Let $X$ be a topological
space with a prebornology. Let $x\in X$. The following are equivalent:
\begin{enumerate}
\item $X$ is locally bounded at $x$;
\item $\mu_{X}\left(x\right)\subseteq G_{X}\left(x\right)$.
\end{enumerate}
\end{thm}
\begin{proof}
$\left(1\right)\Rightarrow\left(2\right)$: There exists a bounded
open neighbourhood $N$ of $x$. By the nonstandard characterisation
of openness (cf. \cite[Theorem 4.1.4]{Rob66}), we know that $\mu_{X}\left(x\right)\subseteq\prescript{\ast}{}{N}$.
By \prettyref{prop:nonst-characterisation-of-boundedness}, $\prescript{\ast}{}{N}\subseteq G_{X}\left(x\right)$.
Hence $\mu_{X}\left(x\right)\subseteq G_{X}\left(x\right)$.

$\left(2\right)\Rightarrow\left(1\right)$: By \prettyref{lem:bornological-approximation-lemma},
there exists a $B\in\prescript{\ast}{}{\mathcal{B}_{X}}$  such that
$G_{X}\left(x\right)\subseteq B$. By \cite[Theorem 4.1.2]{Rob66},
there exists a $U\in\prescript{\ast}{}{\mathcal{O}_{X}}$ such that
$x\in U\subseteq\mu_{X}\left(x\right)$. We have that $U\subseteq B$.
Hence $B$ is an internal bounded neighbourhood of $x$. By transfer,
$x$ has a bounded neighbourhood.
\end{proof}
\begin{cor}[{cf. Stroyan and Luxemburg \cite[Theorem 8.3.14]{SL76}}]
\label{cor:nonst-characterisation-of-local-compactness}Let $X$
be a topological space equipped with the compact bornology. The following
are equivalent:
\begin{enumerate}
\item $X$ is locally compact;
\item $G_{X}\left(x\right)=\mathrm{NS}\left(X\right)$ for all $x\in X$;
\item $\mathrm{FIN}\left(X\right)=\mathrm{NS}\left(X\right)$.
\end{enumerate}
\end{cor}

\begin{cor}
Let $X$ and $Y$ be locally compact topological spaces with the compact
bornologies. Let $f\colon X\to Y$ be a map.
\begin{enumerate}
\item $f$ is bornological if and only if $\prescript{\ast}{}{f}$ maps
nearstandard points to nearstandard points.
\item $f$ is proper if and only if $\prescript{\ast}{}{f}$ maps remote
points to remote points.
\end{enumerate}
\end{cor}
\begin{defn}[Standard]
Let $X$ be a topological space and let $Y$ be a bornological space.
A map $f\colon X\to Y$ is \emph{locally bounded} at $x\in X$ if
there is a neighbourhood $N$ of $x$ such that $f\left(N\right)$
is bounded.
\end{defn}
We can easily obtain the following characterisation of this notion.
\begin{thm}
Let $X,Y,f,x$ be the same as above. The following are equivalent:
\begin{enumerate}
\item $f$ is locally bounded at $x$;
\item $\prescript{\ast}{}{f}\left(\mu_{X}\left(x\right)\right)\subseteq G_{Y}\left(f\left(x\right)\right)$.
\end{enumerate}
\end{thm}
\begin{cor}[Standard]
Let $X$ and $Y$ be topological spaces. Suppose that $Y$ is equipped
with a compatible bornology. If a map $f\colon X\to Y$ is continuous
at $x\in X$, then $f$ is locally bounded at $x$. 
\end{cor}
\begin{proof}
By the nonstandard characterisation of continuity, we know that $\prescript{\ast}{}{f}\left(\mu_{X}\left(x\right)\right)\subseteq\mu_{Y}\left(f\left(x\right)\right)$.
By \prettyref{thm:compatibility-of-top-and-bor}, we have that $\prescript{\ast}{}{f}\left(\mu_{X}\left(x\right)\right)\subseteq\mu_{Y}\left(f\left(x\right)\right)\subseteq G_{Y}\left(f\left(x\right)\right)$.
Thus the above characterisation completes the proof.
\end{proof}

\subsection{Vector topology and vector bornology}

Throughout this subsection, we let the underlying field $\mathbb{K}=\mathbb{R}\text{ or }\mathbb{C}$, and assume that $\mathbb{K}$ is equipped with the usual metric topology and the bounded bornology (see \prettyref{exa:bounded-bornology}).
\begin{example}[Standard]
A vector space $X$ over $\mathbb{K}$ equipped with a topology is
called a \emph{topological vector space} if the addition and the scalar
multiplication are both continuous. A subset $B$ of $X$ is called
\emph{von Neumann bounded} if for each open neighbourhood $U$ of
$0$ there is an $r\in\mathbb{K}$ such that $B\subseteq rU$.
\end{example}
\begin{thm}
\label{thm:nonst-char-vN-bounded}Let $X$ be a topological vector
space. For each subset $B$ of $X$, the following are equivalent:
\begin{enumerate}
\item $B$ is von Neumann bounded;
\item there is an $r\in\prescript{\ast}{}{\mathbb{K}}$ such that $\prescript{\ast}{}{B}\subseteq r\mu_{X}\left(0\right)$.
\end{enumerate}
\end{thm}
\begin{proof}
$\left(1\right)\Rightarrow\left(2\right)$: Using \cite[Theorem 4.1.2]{Rob66}
we can find an internal open neighbourhood $U$ of $0$ such that
$U\subseteq\mu_{X}\left(0\right)$. By transfer, there is an $r\in\prescript{\ast}{}{\mathbb{K}}$
such that $\prescript{\ast}{}{B}\subseteq rU$. Hence $\prescript{\ast}{}{B}\subseteq r\mu_{X}\left(0\right)$.

$\left(2\right)\Rightarrow\left(1\right)$: Let $U$ be an arbitrary
open neighbourhood of $0$. By the nonstandard characterisation of
openness, we have that $\mu_{X}\left(0\right)\subseteq\prescript{\ast}{}{U}$.
Hence $\prescript{\ast}{}{B}\subseteq r\prescript{\ast}{}{U}$. By
transfer, there is an $r'\in\mathbb{K}$ such that $B\subseteq r'U$.
Hence $B$ is von Neumann bounded.
\end{proof}
\begin{defn}[Standard]
A vector space $X$ over $\mathbb{K}$ equipped with a bornology
is called a\emph{ bornological vector space} if the addition and the
scalar multiplication are both bornological. A subset $U$ of $X$
is called \emph{bornivorous} if for each bounded set $B$ there is
an $r\in\mathbb{K}$ such that $B\subseteq rU$.
\end{defn}
\begin{thm}
\label{thm:nonst-char-of-bornivorous}Let $X$ be a bornological vector
space. For each subset $U$ of $X$, the following are equivalent:
\begin{enumerate}
\item $U$ is bornivorous;
\item there is an $r\in\prescript{\ast}{}{\mathbb{K}}$ such that $G_{X}\left(0\right)\subseteq r\prescript{\ast}{}{U}$.
\end{enumerate}
\end{thm}
\begin{proof}
$\left(1\right)\Rightarrow\left(2\right)$: By \nameref{lem:bornological-approximation-lemma},
there exists an internal bounded set $B$ such that $G_{X}\left(0\right)\subseteq B$.
By transfer, there is an $r\in\prescript{\ast}{}{\mathbb{K}}$ such
that $B\subseteq r\prescript{\ast}{}{U}$. Hence $G_{X}\left(0\right)\subseteq r\prescript{\ast}{}{U}$.

$\left(2\right)\Rightarrow\left(1\right)$: Let $B$ be an arbitrary
bounded set. Obviously $\prescript{\ast}{}{B}\subseteq G_{X}\left(0\right)$,
so $\prescript{\ast}{}{B}\subseteq r\prescript{\ast}{}{U}$. By transfer,
there is an $r'\in\mathbb{K}$ such that $B\subseteq r'U$. Hence
$U$ is bornivorous, because $B$ was arbitrary.
\end{proof}
For each topological vector space, the set of all von Neumann bounded
sets is a bornology compatible with the vector space structure. Dually,
for each bornological vector space, the set of all bornivorous sets
makes the space a topological vector space in which the bornivorous
sets are the neighbourhoods of $0$. We omit the proofs and refer
the reader to \cite[p.21]{HN77} and \cite[1.E.8]{HN77}.
\begin{prop}[Standard]
Let $X$ and $Y$ be topological vector spaces. Every continuous
linear map $f\colon X\to Y$ is bornological with respect to the von
Neumann bornologies.
\end{prop}
\begin{proof}
Let $B$ be an arbitrary von Neumann bounded subset of $X$. By \prettyref{thm:nonst-char-vN-bounded}
choose an $r\in\prescript{\ast}{}{\mathbb{K}}$ such that $\prescript{\ast}{}{B}\subseteq r\mu_{X}\left(0\right)$.
By the nonstandard characterisation of continuity, $\prescript{\ast}{}{f}\left(\prescript{\ast}{}{B}\right)\subseteq\prescript{\ast}{}{f}\left(r\mu_{X}\left(0\right)\right)=r\prescript{\ast}{}{f}\left(\mu_{X}\left(0\right)\right)\subseteq r\mu_{Y}\left(0\right)$.
By \prettyref{thm:nonst-char-vN-bounded} $f\left(B\right)$ is von
Neumann bounded. Hence $f$ is bornological.
\end{proof}
\begin{prop}[Standard]
Let $X$ and $Y$ be bornological vector spaces. Every bornological
linear map $f\colon X\to Y$ is continuous with respect to the bornivorous
topologies.
\end{prop}
\begin{proof}
Let $V$ be an arbitrary bornivorous subset of $Y$. By \prettyref{thm:nonst-char-of-bornivorous}
choose an $r\in\prescript{\ast}{}{\mathbb{K}}$ such that $G_{Y}\left(0\right)\subseteq r\prescript{\ast}{}{V}$.
By \prettyref{thm:nonst-characterisation-of-bornological-map}, $G_{X}\left(0\right)\subseteq\prescript{\ast}{}{f}^{-1}\left(G_{Y}\left(0\right)\right)\subseteq\prescript{\ast}{}{f}^{-1}\left(r\prescript{\ast}{}{V}\right)=r\prescript{\ast}{}{f}^{-1}\left(\prescript{\ast}{}{V}\right)$.
By \prettyref{thm:nonst-char-of-bornivorous}, $f^{-1}\left(V\right)$
is bornivorous. Hence $f$ is continuous (at $0$).
\end{proof}

\section{\label{sec:Coarse-structure}Coarse structure}

In this section, we extend nonstandard treatments to coarse structures
in a similar way to \prettyref{sec:Bornology}. As an application,
we give a nonstandard characterisation of the slowly oscillation property,
an important notion in the study of large-scale geometry, and use
it to prove some fundamental facts about that property.

\subsection{Coarse structures and finite closeness relations}
\begin{defn}[Standard]
\label{def:def-of-coarse}A \emph{coarse structure} on a set $X$
is a subset $\mathcal{C}_{X}$ of $\mathcal{P}\left(X\times X\right)$
such that
\begin{enumerate}
\item $\mathcal{C}_{X}$ contains the diagonal set of $X\times X$: $\Delta_{X}\in\mathcal{C}_{X}$;
\item $\mathcal{C}_{X}$ is downward closed: $E\in\mathcal{C}_{X}$ and
$F\subseteq E$ implies $F\in\mathcal{C}_{X}$;
\item $\mathcal{C}_{X}$ is closed under finite union: $E,F\in\mathcal{C}_{X}$
implies $E\cup F\in\mathcal{C}_{X}$;
\item $\mathcal{C}_{X}$ is closed under composition: $E,F\in\mathcal{C}_{X}$
implies $E\circ F\in\mathcal{C}_{X}$;
\item $\mathcal{C}_{X}$ is closed under inversion: $E\in\mathcal{C}_{X}$
implies $E^{-1}\in\mathcal{C}_{X}$.
\end{enumerate}
The pair $\left(X,\mathcal{C}_{X}\right)$ is called a \emph{coarse
space}. The elements of $\mathcal{C}_{X}$ are called \emph{controlled
sets} of $X$.
\end{defn}
For instance, consider a metric space $X$. Then the collection of
all binary relations $E$ on $X$ with $\sup_{\left(x,y\right)\in E}d_{X}\left(x,y\right)<\infty$
gives a coarse structure on $X$. See \prettyref{subsec:Examples-of-coarse}
for further examples.
\begin{defn}
Let $X$ be a coarse space. We say that $x,y\in\prescript{\ast}{}{X}$
are \emph{finitely close}, denoted by $x\sim_{X}y$, if there exists
an $E\in\mathcal{C}_{X}$ such that $\left(x,y\right)\in\prescript{\ast}{}{E}$.
Equivalently, the \emph{finite closeness relation} $\sim_{X}$ is
defined as the union $\bigcup_{E\in\mathcal{C}_{X}}\prescript{\ast}{}{E}$.
\end{defn}
The coarse structure can be recovered from its finite closeness relation
$\sim_{X}$. To see this, we use the following lemma, the uniform
counterpart of \nameref{lem:bornological-approximation-lemma}.
\begin{lem}[Coarse Approximation Lemma]
\label{lem:coarse-approximation-lemma}Let $X$ be a coarse space.
There exists an $E\in\prescript{\ast}{}{\mathcal{C}_{X}}$ with ${\sim_{X}}\subseteq E$.
\end{lem}
\begin{proof}
Since $\mathcal{C}_{X}$ is closed under finite union, for every finite
subset $\mathcal{F}$ of $\mathcal{C}_{X}$ there exists an $E'\in\mathcal{C}_{X}$
such that $F\subseteq E'$ for all $F\in\mathcal{F}$. By weak saturation,
there exists an $E\in\prescript{\ast}{}{\mathcal{C}_{X}}$ such that
$\prescript{\ast}{}{F}\subseteq E$ for all $F\in\mathcal{C}_{X}$.
It follows that ${\sim_{X}}\subseteq E$.
\end{proof}
\begin{prop}
\label{prop:nonst-characterisation-of-controlled-sets}Let $X$ be
a coarse structure. Let $E$ be a subset of $X\times X$. The following
are equivalent:
\begin{enumerate}
\item $E\in\mathcal{C}_{X}$;
\item $\prescript{\ast}{}{E}\subseteq{\sim_{X}}$.
\end{enumerate}
\end{prop}
\begin{proof}
$\left(1\right)\Rightarrow\left(2\right)$ is trivial by the definition
of $\sim_{X}$. Let us prove $\left(2\right)\Rightarrow\left(1\right)$.
By \nameref{lem:coarse-approximation-lemma}, there exists an $F\in\prescript{\ast}{}{\mathcal{C}_{X}}$
such that ${\sim_{X}}\subseteq F$. Since $\prescript{\ast}{}{E}\subseteq{\sim_{X}}\subseteq F$,
we have that $E\in\mathcal{C}_{X}$ by transfer.
\end{proof}
\begin{prop}
For any coarse space $X$, $\sim_{X}$ is a galactic equivalence relation
on $\prescript{\ast}{}{X}$.
\end{prop}
\begin{proof}
By definition, $\sim_{X}$ is galactic. Since $\Delta_{X}\in\mathcal{C}_{X}$,
we have that $\Delta_{\prescript{\ast}{}{X}}=\prescript{\ast}{}{\Delta_{X}}\subseteq{\sim_{X}}$.
The composability of $\mathcal{C}_{X}$ implies that 
\[
{\sim_{X}}\circ{\sim_{X}}=\bigcup_{E\in\mathcal{C}_{X}}\prescript{\ast}{}{E}\circ\bigcup_{F\in\mathcal{C}_{X}}\prescript{\ast}{}{F}=\bigcup_{E,F\in\mathcal{C}_{X}}\prescript{\ast}{}{\left(E\circ F\right)}\subseteq{\sim_{X}}.
\]
The invertibility of $\mathcal{C}_{X}$ implies that
\[
{\sim_{X}}^{-1}=\left(\bigcup_{E\in\mathcal{C}_{X}}\prescript{\ast}{}{E}\right)^{-1}=\bigcup_{E\in\mathcal{C}_{X}}\prescript{\ast}{}{E}^{-1}\subseteq{\sim_{X}}.
\]
\end{proof}
\begin{thm}
\label{thm:G-equiv-rel-determines-coarse-str}Let $X$ be any set.
If $R$ is a galactic equivalence relation on $\prescript{\ast}{}{X}$,
then $X$ admits a unique coarse structure whose finite closeness
coincides with $R$.
\end{thm}
\begin{proof}
Let $\mathcal{C}_{X}=\set{E\subseteq X\times X|\prescript{\ast}{}{E}\subseteq R}$.
Clearly $\mathcal{C}_{X}$ satisfies (2) and (3) of \prettyref{def:def-of-coarse}.
The reflexivity, transitivity and symmetricity of $R$ imply (1),
(4) and (5), respectively. Hence $\mathcal{C}_{X}$ is a coarse structure
on $X$. Since $R$ is galactic, $R=\bigcup_{E\in\mathcal{C}_{X}}\prescript{\ast}{}{E}={\sim_{X}}$.
\end{proof}
Hence the correspondence $\mathcal{C}_{X}\leftrightarrow{\sim_{X}}$
is a bijection between the coarse structures on $X$ and the galactic
equivalence relations on $\prescript{\ast}{}{X}$.
\begin{rem}
These two results are the large-scale counterpart of \cite[Theorem 1.7]{CS83}
which states that monadic equivalence relations correspond to uniformities.
Here a set $A$ is said to be \emph{monadic} if $A=\bigcap\set{\prescript{\ast}{}{B}|A\subseteq\prescript{\ast}{}{B}}$.
\end{rem}
\begin{defn}[Standard]
Let $X$ be a coarse space. A subset $B$ of $X$ is said to be \emph{bounded}
if $B\times B\in\mathcal{C}_{X}$ holds. The \emph{induced prebornology}
on $X$ is $\mathcal{B}_{X}=\set{B\subseteq X|B\text{ is bounded}}$.
\end{defn}
\begin{rem}
\label{rem:Adjoint}Notice that $\mathcal{B}_{X}$ may not be a bornology,
while it always is a prebornology. The construction of induced prebornology
can be considered as a (forgetful) functor $U$ from the category
of coarse spaces to the category of prebornological spaces. Conversely,
given a prebornological space $X$, define a galactic equivalence
relation $\sim_{X}$ on $\prescript{\ast}{}{X}$ by
\[
x\sim_{X}y\iff x=y\text{ or }x,y\in G_{X}\left(z\right)\text{ for some }z\in X.
\]
Then $\sim_{X}$ determines a unique coarse structure on $X$ by \prettyref{thm:G-equiv-rel-determines-coarse-str}.
This construction can be extended to a functor $T$ from the category
of prebornological spaces to the category of coarse spaces. It is
easy to see that $UTX=X$. Moreover, it can be proved that $T$ is
the left adjoint of $U$.
\end{rem}
\begin{prop}
\label{prop:nonst-characterisation-of-boundedness-wrt-ind-bornology}Let
$X$ be a coarse space and let $B$ be a subset of $X$. The following
are equivalent:
\begin{enumerate}
\item $B$ is bounded;
\item $x\sim_{X}y$ holds for any $x,y\in\prescript{\ast}{}{B}$.
\end{enumerate}
\end{prop}
\begin{proof}
(2) is equivalent to ``$\prescript{\ast}{}{B}\times\prescript{\ast}{}{B}\subseteq{\sim_{X}}$''.
By \prettyref{prop:nonst-characterisation-of-controlled-sets}, this
is also equivalent to ``$B\times B\in\mathcal{C}_{X}$''.
\end{proof}
\begin{defn}
Let $X$ be a coarse space. The \emph{coarse galaxy} of $x\in\prescript{\ast}{}{X}$
is the set
\[
G_{X}^{c}\left(x\right)=\set{y\in\prescript{\ast}{}{X}|x\sim_{X}y}.
\]
\end{defn}
The coarse galaxy $G_{X}^{c}$ is defined on $\prescript{\ast}{}{X}$,
while the galaxy $G_{X}\left(x\right)$ is defined only on $X$.
\begin{prop}
\label{prop:galaxy=00003Dcoarse-galaxy}Let $X$ be a coarse space
equipped with the induced prebornology. Then, $G_{X}^{c}\left(x\right)=G_{X}\left(x\right)$
holds for all $x\in X$.
\end{prop}
\begin{proof}
Let $E\in\mathcal{C}_{X}$. Consider the $E$-neighbourhood of $x$,
$E\left[x\right]=\set{y\in X|\left(x,y\right)\in E}$. For each $y,y'\in\prescript{\ast}{}{\left(E\left[x\right]\right)}$,
since $\left(x,y\right)\in\prescript{\ast}{}{E}$ and $\left(x,y'\right)\in\prescript{\ast}{}{E}$,
we have that $y\sim_{X}x\sim_{X}y'$. By \prettyref{prop:nonst-characterisation-of-boundedness-wrt-ind-bornology},
$E\left[x\right]$ is bounded. The union $E\left[x\right]\cup\set{x}=\left(E\cup\Delta_{X}\right)\left[x\right]$
is also bounded. Hence
\begin{align*}
G_{X}^{c}\left(x\right) & =\bigcup_{E\in\mathcal{C}_{X}}\prescript{\ast}{}{E}\left[x\right]\\
 & \subseteq\bigcup_{E\in\mathcal{C}_{X}}\prescript{\ast}{}{E\left[x\right]}\cup\set{x}\\
 & \subseteq\bigcup_{B\in\mathcal{BN}_{X}\left(x\right)}\prescript{\ast}{}{B}\\
 & =G_{X}\left(x\right).
\end{align*}

Let $B$ be a bounded set containing $x$. Then, $\set{x}\times B\subseteq B\times B\in\mathcal{C}_{X}$,
so $\set{x}\times B\in\mathcal{C}_{X}$. This implies that
\begin{align*}
G_{X}\left(x\right) & =\bigcup_{B\in\mathcal{BN}_{X}\left(x\right)}\prescript{\ast}{}{B}\\
 & =\bigcup_{B\in\mathcal{BN}_{X}\left(x\right)}\left(\set{x}\times\prescript{\ast}{}{B}\right)\left[x\right]\\
 & \subseteq\bigcup_{E\in\mathcal{C}_{X}}\prescript{\ast}{}{\left(E\left[x\right]\right)}\\
 & =G_{X}^{c}\left(x\right).\\
\end{align*}
\end{proof}
\begin{cor}
Let $X$ be a coarse space. The following are equivalent:
\begin{enumerate}
\item $X$ is connected (as a prebornological space);
\item $x\sim_{X}y$ for all $x,y\in X$.
\end{enumerate}
\end{cor}
\begin{proof}
Immediate from \prettyref{prop:characterisation-of-bornological-connectedness}
and \prettyref{prop:galaxy=00003Dcoarse-galaxy}.
\end{proof}

\subsection{\label{subsec:Examples-of-coarse}Examples of coarse spaces}
\begin{example}
Let $X$ be a set. The \emph{maximal coarse structure} on $X$ is
$\mathcal{P}\left(X\times X\right)$. The finite closeness relation
is $\prescript{\ast}{}{X}\times\prescript{\ast}{}{X}$.
\end{example}

\begin{example}
Let $X$ be a set. The \emph{discrete coarse structure} on $X$ is
$\mathcal{P}\left(\Delta_{X}\right)$. The finite closeness relation
is the diagonal $\Delta_{\prescript{\ast}{}{X}}$. Generally, given
an equivalence relation $E$ on $X$, its power set $\mathcal{P}\left(E\right)$
is a coarse structure on $X$ whose finite closeness relation is precisely
$\prescript{\ast}{}{E}$.
\end{example}

\begin{example}
Let $X$ be a set. The \emph{finite coarse structure} on $X$ is
\[
\set{E\subseteq X\times X|E\setminus\Delta_{X}\text{ is finite}}.
\]
The finite closeness relation is ${\sim_{X}}=\Delta_{\prescript{\ast}{}{X}}\cup X\times X$:
let $x,y\in\prescript{\ast}{}{X}$. Clearly, if $x=y$ or $x,y\in X$,
then $x\sim_{X}y$. To see the reverse direction, suppose that $x\sim_{X}y$
but $x\neq y$. Choose a controlled set $E$ such that $\left(x,y\right)\in\prescript{\ast}{}{E}$.
Since $E$ contains only a \emph{standard} finite number of pairs
$\left(x_{0},y_{0}\right),\ldots,\left(x_{n},y_{n}\right)$ off the
diagonal $\Delta_{X}$, we have that $\left(x,y\right)=\left(x_{i},y_{i}\right)$
for some $0\le i\leq n$ by transfer. It follows that $x,y\in X$.
\end{example}

\begin{example}
Let $X$ be a pseudometric space. The \emph{bounded coarse structure}
of $X$ is
\[
\Set{E\subseteq X\times X|\sup_{\left(x,y\right)\in E}d_{X}\left(x,y\right)<\infty}.
\]
The finite closeness relation is
\[
x\sim_{X}y\iff\prescript{\ast}{}{d_{X}}\left(x,y\right)\text{ is finite}.
\]
This construction works even when the metric function is allowed to
take the value $+\infty$.
\end{example}

\begin{example}
Let $\varGamma$ be a bornological group, i.e. a group together with a bornology such that the group operations are bornological. The \emph{left coarse structure}
of $\varGamma$ is the coarse structure generated by the sets of the
form $\set{\left(x,y\right)\in\varGamma\times\varGamma|x^{-1}y\in B}$,
where $B\in\mathcal{B}_{\varGamma}$. The finite closeness relation
is
\[
x\sim_{\varGamma,l}y\iff x\backslash y\left(=x^{-1}y\right)\text{ is finite}.
\]
Similarly, the \emph{right coarse structure} of $\varGamma$ is the
coarse structure generated by the sets of the form $\set{\left(x,y\right)\in\varGamma\times\varGamma|xy^{-1}\in B}$,
where $B\in\mathcal{B}_{\varGamma}$. The finite closeness relation
is
\[
x\sim_{\varGamma,r}y\iff x/y\left(=xy^{-1}\right)\text{ is finite}.
\]
If $\varGamma$ is commutative, the left and the right structures
coincide.
\end{example}

\begin{example}
Let $X$ be a coarse space and $A$ a subset of $X$. The \emph{subspace
coarse structure} of $A$ is the restriction $\mathcal{C}_{X}\restriction A=\set{E\cap A^{2}|E\in\mathcal{C}_{X}}$.
The finite closeness relation $\sim_{A}$ is equal to the restriction
${\sim_{X}}\cap\prescript{\ast}{}{A}^{2}$.
\end{example}

\begin{example}
Let $X$ and $Y$ be coarse spaces. The \emph{product coarse structure}
on $X\times Y$ is
\[
\mathcal{C}_{X\times Y}=\set{E\subseteq\left(X\times Y\right)^{2}|\pi_{X^{2}}\left(E\right)\in\mathcal{C}_{X}\text{ and }\pi_{Y^{2}}\left(E\right)\in\mathcal{C}_{Y}},
\]
where $\pi_{X^{2}}\colon\left(X\times Y\right)^{2}\to X^{2}$ and
$\pi_{Y^{2}}\colon\left(X\times Y\right)^{2}\to Y^{2}$ are the canonical
projections. The finite closeness relation is
\[
\left(x,y\right)\sim_{X\times Y}\left(x',y'\right)\iff x\sim_{X}x'\text{ and }y\sim_{Y}y'.
\]
\end{example}

\begin{example}
Let $\set{X_{i}}_{i\in I}$ be a family of coarse spaces. The \emph{sum
coarse structure} on $S=\coprod_{i\in I}X_{i}$ is $\mathcal{C}_{S}=\bigcup_{i\in I}\mathcal{C}_{X_{i}}$.
The finite closeness relation is
\[
u\sim_{S}v\iff u=v,\text{ or }u,v\in\prescript{\ast}{}{X_{i}}\text{ and }u\sim_{X_{i}}v\text{ for some }i\in I.
\]
\end{example}

\subsection{Bornologous maps}
\begin{defn}[Standard]
Let $X$ and $Y$ be coarse spaces. A map $f\colon X\to Y$ is said
to be \emph{bornologous} if $\set{\left(f\left(x\right),f\left(y\right)\right)|\left(x,y\right)\in E}\in\mathcal{C}_{Y}$
holds for any $E\in\mathcal{C}_{X}$.
\end{defn}
\begin{thm}
\label{thm:nonst-characterisation-of-bornologous}Let $X$ and $Y$
be coarse spaces and let $f\colon X\to Y$ be a map. The following
are equivalent:

\begin{enumerate}
\item $f$ is bornologous;
\item for any $x,y\in\prescript{\ast}{}{X}$, $x\sim_{X}y$ implies $\prescript{\ast}{}{f}\left(x\right)\sim_{Y}\prescript{\ast}{}{f}\left(y\right)$;
\item for any $x\in\prescript{\ast}{}{X}$, $\prescript{\ast}{}{f}\left(G_{X}^{c}\left(x\right)\right)\subseteq G_{Y}^{c}\left(\prescript{\ast}{}{f}\left(x\right)\right)$.
\end{enumerate}
\end{thm}
\begin{proof}
$\left(1\right)\Rightarrow\left(2\right)$: Let $E\in\mathcal{C}_{X}$
be such that $\left(x,y\right)\in\prescript{\ast}{}{E}$. Since $f$
is bornologous, $F=\set{\left(f\left(x'\right),f\left(y'\right)\right)|\left(x',y'\right)\in E}$
is in $\mathcal{C}_{Y}$. By transfer, we have $\left(\prescript{\ast}{}{f}\left(x\right),\prescript{\ast}{}{f}\left(y\right)\right)\in\prescript{\ast}{}{F}$.
Hence $\prescript{\ast}{}{f}\left(x\right)\sim_{Y}\prescript{\ast}{}{f}\left(y\right)$.

$\left(2\right)\Rightarrow\left(1\right)$: Let $E\in\mathcal{C}_{X}$.
By \nameref{lem:coarse-approximation-lemma}, there exists an $F\in\prescript{\ast}{}{\mathcal{C}_{Y}}$
with $\left(\sim_{Y}\right)\subseteq F$. For each $\left(x,y\right)\in\prescript{\ast}{}{E}$,
since $x\sim_{X}y$, we have that $\prescript{\ast}{}{f}\left(x\right)\sim_{Y}\prescript{\ast}{}{f}\left(y\right)$,
so $\left(\prescript{\ast}{}{f}\left(x\right),\prescript{\ast}{}{f}\left(y\right)\right)\in F$.
Hence $\set{\left(f\left(x\right),f\left(y\right)\right)|\left(x,y\right)\in E}\in\mathcal{C}_{Y}$
by transfer.

$\left(2\right)\Leftrightarrow\left(3\right)$: Trivial.
\end{proof}
\begin{prop}[Standard]
Every bornologous map $f\colon X\to Y$ is bornological.
\end{prop}
\begin{proof}
Let $x\in X$. Then, by \prettyref{prop:galaxy=00003Dcoarse-galaxy}
and \prettyref{thm:nonst-characterisation-of-bornologous}, $\prescript{\ast}{}{f}\left(G_{X}\left(x\right)\right)=\prescript{\ast}{}{f}\left(G_{X}^{c}\left(x\right)\right)\subseteq G_{Y}^{c}\left(f\left(x\right)\right)=G_{Y}\left(f\left(x\right)\right)$.
By \prettyref{thm:nonst-characterisation-of-bornological-map}, $f$
is bornological at $x$.
\end{proof}
\begin{defn}[Standard]
Let $X$ be any set and $Y$ a coarse space. Two maps $f,g\colon X\to Y$
are said to be \emph{bornotopic} if $\set{\left(f\left(x\right),g\left(x\right)\right)|x\in X}\in\mathcal{C}_{Y}$.
\end{defn}
\begin{prop}[Standard]
Let $X$ and $Y$ be coarse spaces such that $X$ is connected. Let
$f\colon X\to Y$ and $g\colon Y\to X$.
\begin{enumerate}
\item If $g$ is bornologous, and if $g\circ f$ is bornotopic to $\mathrm{id}_{X}$,
then $f$ is proper.
\item If $f$ is proper, and if $f\circ g$ is bornotopic to $\mathrm{id}_{Y}$,
then $g$ is bornological.
\end{enumerate}
\end{prop}
\begin{proof}
We only prove (1), because (2) can be proved dually. Let $x\in\mathrm{INF}\left(X\right)$.
Since $g\circ f$ is bornotopic to $\mathrm{id}_{X}$, we have that
$\prescript{\ast}{}{g}\left(\prescript{\ast}{}{f}\left(x\right)\right)\sim_{X}x$.
It follows that $\prescript{\ast}{}{g}\left(\prescript{\ast}{}{f}\left(x\right)\right)\in\mathrm{INF}\left(X\right)$.
Since $g$ is bornological, and by \prettyref{thm:nonst-characterisations-of-bornological-and-proper},
it must hold that $\prescript{\ast}{}{f}\left(x\right)\in\mathrm{INF}\left(Y\right)$.
Again by \prettyref{thm:nonst-characterisations-of-bornological-and-proper},
$f$ is proper.
\end{proof}

\subsection{Equibornologous families}
\begin{defn}[Standard]
Let $X$ and $Y$ be coarse spaces. Let $\mathcal{F}$ be a subset
of $Y^{X}$. For $E\subseteq X\times X$, we denote $\mathcal{F}\left(E\right)=\set{\left(f\left(x\right),f\left(y\right)\right)|f\in\mathcal{F},\left(x,y\right)\in E}$.
$\mathcal{F}$ is said to be\emph{ equibornologous} if $\mathcal{F}\left(E\right)\in\mathcal{C}_{Y}$
for all $E\in\mathcal{C}_{X}$.
\end{defn}
\begin{thm}
Let $X$ and $Y$ be coarse spaces and let $\mathcal{F}\subseteq Y^{X}$.
The following are equivalent:
\begin{enumerate}
\item $\mathcal{F}$ is equibornologous;
\item for any $f\in\prescript{\ast}{}{\mathcal{F}}$ and $x,y\in\prescript{\ast}{}{X}$,
$x\sim_{X}y$ implies $f\left(x\right)\sim_{Y}f\left(y\right)$.
\end{enumerate}
\end{thm}
\begin{proof}
$\left(1\right)\Rightarrow\left(2\right)$: By \prettyref{prop:nonst-characterisation-of-controlled-sets}
there exists a controlled set $E$ of $X$ such that $\left(x,y\right)\in\prescript{\ast}{}{E}$.
By transfer, $\left(f\left(x\right),f\left(y\right)\right)\in\prescript{\ast}{}{\mathcal{F}}\left(B\right)$
holds. Since $\mathcal{F}\left(E\right)\in\mathcal{C}_{Y}$, by \prettyref{prop:nonst-characterisation-of-controlled-sets},
$f\left(x\right)\sim_{Y}f\left(y\right)$ holds.

$\left(2\right)\Rightarrow\left(1\right)$: Suppose that $\mathcal{F}$
is not equibornologous. There exists a controlled set $E$ of $X$
such that $\mathcal{F}\left(E\right)$ is not a controlled set of
$Y$. By \prettyref{prop:nonst-characterisation-of-controlled-sets},
there are $\left(y,y'\right)\in\prescript{\ast}{}{\left(\mathcal{F}\left(E\right)\right)}$
such that $y\nsim_{Y}y'$. By transfer, $\left(y,y'\right)=\left(f\left(x\right),f\left(x'\right)\right)$
holds for some $f\in\prescript{\ast}{}{\mathcal{F}}$ and $\left(x,x'\right)\in\prescript{\ast}{}{E}$.
The latter implies that $x\sim_{X}x'$.
\end{proof}

\subsection{Compatibility of uniform and coarse structures}
\begin{defn}[Standard]
Let $X$ be a set. A uniformity $\mathcal{U}_{X}$ and a coarse structure
$\mathcal{C}_{X}$ on $X$ are said to be \emph{compatible} if $X$
is uniformly locally bounded, i.e., $\mathcal{U}_{X}\cap\mathcal{C}_{X}\neq\varnothing$.
\end{defn}
\begin{thm}
\label{thm:compatibility-of-unif-and-coarse}Let $X$ be a uniform
space with a coarse structure. The following are equivalent:
\begin{enumerate}
\item $X$ is uniformly locally bounded;
\item $x\approx_{X}y$ implies $x\sim_{X}y$ for all $x,y\in\prescript{\ast}{}{X}$.
\end{enumerate}
\end{thm}
\begin{proof}
$\left(1\right)\Rightarrow\left(2\right)$: Let $U\in\mathcal{U}_{X}\cap\mathcal{C}_{X}$.
By the nonstandard characterisation of entourages (cf. \cite[Observation 8.4.25]{SL76}),
we know that ${\approx_{X}}\subseteq\prescript{\ast}{}{U}$. By \prettyref{prop:nonst-characterisation-of-controlled-sets},
$\prescript{\ast}{}{U}\subseteq{\sim_{X}}$. Hence $\approx_{X}$
implies $\sim_{X}$.

$\left(2\right)\Rightarrow\left(1\right)$: By \nameref{lem:coarse-approximation-lemma},
there exists an $E\in\prescript{\ast}{}{\mathcal{C}_{X}}$ such that
${\sim_{X}}\subseteq E$. By weak saturation, there exists an $U\in\prescript{\ast}{}{\mathcal{U}_{X}}$
such that $U\subseteq{\approx_{X}}$. We have that $U\subseteq E$.
By transfer, $E\in\prescript{\ast}{}{\mathcal{U}_{X}}\cap\prescript{\ast}{}{\mathcal{C}_{X}}$,
so $\mathcal{U}_{X}\cap\mathcal{C}_{X}\neq\varnothing$.
\end{proof}

\subsection{Slowly oscillating maps and Higson functions}
\begin{defn}[Standard]
Let $X$ be a connected coarse space and let $Y$ be a uniform space.
A map $\varphi\colon X\to Y$ is said to be \emph{slowly oscillating}
if for every $E\in\mathcal{C}_{X}$ and for every $U\in\mathcal{U}_{Y}$
there is a $B\in\mathcal{B}_{X}$ such that $\left(\varphi\left(x\right),\varphi\left(y\right)\right)\in U$
holds for all $\left(x,y\right)\in E\setminus\left(B\times B\right)$.
\end{defn}
\begin{thm}
\label{thm:nonst-characterisation-of-slow-oscillation}Let $X$, $Y$
and $\varphi$ be the same as above. The following are equivalent:
\begin{enumerate}
\item $\varphi$ is slowly oscillating;
\item for any $x,y\in\mathrm{INF}\left(X\right)$, if $x\sim_{X}y$, then
$\prescript{\ast}{}{\varphi}\left(x\right)\approx_{Y}\prescript{\ast}{}{\varphi}\left(y\right)$.
\end{enumerate}
\end{thm}
\begin{proof}
$\left(1\right)\Rightarrow\left(2\right)$: Let $U\in\mathcal{U}_{Y}$.
Choose an $E\in\mathcal{C}_{X}$ such that $\left(x,y\right)\in\prescript{\ast}{}{E}$.
There is a $B\in\mathcal{B}_{X}$ such that $\left(\varphi\left(x'\right),\varphi\left(y'\right)\right)\in U$
for all $\left(x',y'\right)\in E\setminus\left(B\times B\right)$.
Since $x,y\in\mathrm{INF}\left(X\right)$, we have that $x,y\notin\prescript{\ast}{}{B}$.
By transfer, $\left(\varphi\left(x\right),\varphi\left(y\right)\right)\in\prescript{\ast}{}{U}$
holds. Hence $\prescript{\ast}{}{\varphi}\left(x\right)\approx_{Y}\prescript{\ast}{}{\varphi}\left(y\right)$,
because $U$ is arbitrary.

$\left(2\right)\Rightarrow\left(1\right)$: Let $E\in\mathcal{C}_{X}$
and $U\in\mathcal{U}_{Y}$. By \nameref{lem:bornological-approximation-lemma},
we can find a $B'\in\prescript{\ast}{}{\mathcal{B}_{X}}$ such that
$\mathrm{FIN}\left(X\right)\subseteq B'$. By assumption, we have
that $\left(\varphi\left(x\right),\varphi\left(y\right)\right)\in{\approx}_{Y}\subseteq\prescript{\ast}{}{U}$
for all $\left(x,y\right)\in\prescript{\ast}{}{E}\setminus\left(B'\times B'\right)\subseteq\mathrm{INF}\left(X\right)\times\mathrm{INF}\left(X\right)$.
By transfer, there is a $B\in\mathcal{B}_{X}$ such that $\left(\varphi\left(x\right),\varphi\left(y\right)\right)\in U$
holds for all $\left(x,y\right)\in E\setminus\left(B\times B\right)$.
Therefore $\varphi$ is slowly oscillating.
\end{proof}
Let us apply this characterisation to prove some fundamental facts
about slowly oscillating maps.
\begin{prop}[Standard]
Let $X$ and $Y$ be connected coarse spaces, $Z$ and $W$ uniform
spaces, $f\colon X\to Y$ a proper bornologous map, $\varphi\colon Y\to Z$
a slowly oscillating map and $g\colon Z\to W$ a uniformly continuous
map. Then, the composition $g\circ\varphi\circ f\colon X\to W$ is
slowly oscillating.
\end{prop}
\begin{proof}
By \prettyref{thm:nonst-characterisations-of-bornological-and-proper}
and \prettyref{thm:nonst-characterisation-of-bornologous}, $f$ sends
$\sim_{X}$-pairs to $\sim_{Y}$-pairs, and maps $\mathrm{INF}\left(X\right)$
to $\mathrm{INF}\left(Y\right)$. By \prettyref{thm:nonst-characterisation-of-slow-oscillation},
$\varphi$ sends $\sim_{Y}$-pairs of infinite points to $\approx_{Z}$-pairs.
By the nonstandard characterisation of uniform continuity (cf. \cite[Theorem 8.4.23]{SL76}),
$g$ sends $\approx_{Z}$-pairs to $\approx_{W}$-pairs. Combining
them, $g\circ\varphi\circ f$ sends $\sim_{X}$-pairs of infinite
points to $\approx_{W}$-pairs. Finally, by \prettyref{thm:nonst-characterisation-of-slow-oscillation},
$g\circ\varphi\circ f$ is slowly oscillating.
\end{proof}
\begin{prop}[Standard]
A uniform limit of slowly oscillating maps is slowly oscillating.
\end{prop}
\begin{proof}
Let $\set{\varphi_{\lambda}\colon X\to Y}_{\lambda\in\Lambda}$ be
a directed sequence of slowly oscillating maps. Suppose that $\set{\varphi_{\lambda}}_{\lambda\in\Lambda}$
is uniformly convergent to a map $\psi\colon X\to Y$. Let $x,y\in\mathrm{INF}\left(X\right)$
and suppose that $x\sim_{X}y$. By \prettyref{thm:nonst-characterisation-of-slow-oscillation},
$\prescript{\ast}{}{\varphi_{\lambda}}\left(x\right)\approx_{Y}\prescript{\ast}{}{\varphi_{\lambda}}\left(y\right)$
holds for any $\lambda\in\Lambda$. Applying Robinson's lemma (cf.
\cite[Theorem 4.3.10]{Rob66} for the metrisable case), we can find
an $\lambda_{0}\in\Lambda_{\infty}$ such that $\prescript{\ast}{}{\varphi_{\lambda_{0}}}\left(x\right)\approx_{Y}\prescript{\ast}{}{\varphi_{\lambda_{0}}}\left(y\right)$,
where $\Lambda_{\infty}=\set{\lambda\in\prescript{\ast}{}{\Lambda}|\Lambda\leq\lambda}$.
By the nonstandard characterisation of uniform convergence (cf. \cite[Theorem 4.6.1]{Rob66}
for the metrisable case), $\prescript{\ast}{}{\psi}\left(x\right)\approx_{Y}\prescript{\ast}{}{\varphi_{\lambda_{0}}}\left(x\right)\approx_{Y}\prescript{\ast}{}{\varphi_{\lambda_{0}}}\left(y\right)\approx_{Y}\prescript{\ast}{}{\psi}\left(y\right)$.
Again by \prettyref{thm:nonst-characterisation-of-slow-oscillation},
$\psi$ is slowly oscillating.
\end{proof}
Note that this proof depends on the \emph{full} saturation principle,
because so does Robinson's lemma.
\begin{defn}[Standard]
Let $X$ be a connected coarse space equipped with a topology. A
$\mathbb{C}$-valued, bounded, slowly oscillating, continuous function
on $X$ is called a \emph{Higson function}.
\end{defn}
\begin{prop}[{Standard; Roe \cite[Lemma 5.3]{Roe93}}]
The class $C_{h}\left(X\right)$ of Higson functions forms a $C^{\ast}$-algebra
with respect to to the pointwise operations.
\end{prop}
\begin{proof}
It has been proved that $C_{h}\left(X\right)$ is complete with respect
to the supremum norm. What we have to prove is that $C_{h}\left(X\right)$
is closed under the pointwise operations of $\ast$-algebra. Here
let us only prove that a product of two Higson functions is again
a Higson function. Let $\varphi,\psi\colon X\to\mathbb{C}$ be Higson
functions. Obviously $\varphi\psi$ is bounded and continuous. Now,
let $x,y\in\mathrm{INF}\left(X\right)$ and suppose that $x\sim_{X}y$.
Since $\varphi$ and $\psi$ are both bounded, $\prescript{\ast}{}{\varphi}\left(x\right)$
and $\prescript{\ast}{}{\psi}\left(y\right)$ are both finite. By
\prettyref{thm:nonst-characterisation-of-slow-oscillation},
\begin{align*}
\prescript{\ast}{}{\left(\varphi\psi\right)}\left(x\right) & =\prescript{\ast}{}{\varphi}\left(x\right)\prescript{\ast}{}{\psi}\left(x\right)\\
 & \approx_{\mathbb{C}}\prescript{\ast}{}{\varphi}\left(x\right)\prescript{\ast}{}{\psi}\left(y\right)\\
 & \approx_{\mathbb{C}}\prescript{\ast}{}{\varphi}\left(y\right)\prescript{\ast}{}{\psi}\left(y\right)\\
 & =\prescript{\ast}{}{\left(\varphi\psi\right)}\left(y\right).
\end{align*}
Note that we have used the algebraic law $\text{infinitesimal}\times\text{finite}=\text{infinitesimal}$
in the second and the third $\approx_{\mathbb{C}}$s. Again by \prettyref{thm:nonst-characterisation-of-slow-oscillation},
$\varphi\psi$ is slowly oscillating. Hence $\varphi\psi$ is a Higson
function.
\end{proof}

\addcontentsline{toc}{section}{References}
\bibliographystyle{elsarticle-num}
\bibliography{bibliography}

\section*{Corrigendum}
\addcontentsline{toc}{section}{Corrigendum}
In \prettyref{prop:equibounded-family-of-linear-maps}, the linearity of (the elements of) $\mathcal{F}$ must be assumed.

\end{document}